\documentclass{article}
\usepackage{graphicx,geometry,amsmath,amsthm,amssymb,mathdots,authblk,hyperref} 
\usepackage{array,float}

\newtheorem{theorem}{Theorem}[section]
\newtheorem{lemma}[theorem]{Lemma}
\newtheorem{corollary}[theorem]{Corollary}
\newtheorem{remark}[theorem]{Remark}

\numberwithin{equation}{section}

\title{Energy relations for the generalised Vladimirov derivative via Bruhat-Tits tree extensions}
\author{An Huang, Yaojia Sun}
\date{}

\begin{document}

\maketitle
\begin{abstract}
    We formulate a proof on non-Archimedean analogue of the Caffarelli–Silvestre extension for fractional Laplacian on the Bruhat–Tits tree $T_p$ in terms of energy relations after identifying $\partial T_p$ with $\mathbb{P}^1(\mathbb{Q}_p)$, which reproduces the generalised Vladimirov derivative of the rescaled boundary function.
\end{abstract}

\section{Introduction}
In 2007, Caffarelli and Silvestre investigated the extension problem for the fractional Laplacian \cite{MR2354493}. They began by considering the harmonic extension of $f(x)$ to the upper half-space $\mathbb{R}^n\times[0,+\infty)$
\begin{align*}
\begin{cases}
    \Delta u(x,r)=0,&(x,r)\in\mathbb{R}^n\times[0,+\infty),\\
    u(x,0)=f(x),&x\in\mathbb{R}^n,
\end{cases}
\end{align*}
and proved that
\begin{align*}
    (-\Delta)^{\frac{1}{2}}f(x)=-u_y(x,0).
\end{align*}
This means that applying the square root of the Laplacian to a function on the boundary is equivalent to taking the normal derivative of its harmonic extension at the boundary.

For the general case, they first considered the Laplacian on a function defined on $(x,y)\in\mathbb{R}^n\times\mathbb{R}^{1+a}$ with rotationally invariant $y$, i.e., $u(x,y)=u(x,|y|)$. Denoting $r=|y|$ and $X=(x,r)$, then $\Delta=\Delta_x+\partial_r^2+\frac{a}{r}\partial_r$. So they considered
\begin{equation}
\label{extensionfunction_a}
\begin{aligned}
\begin{cases}
    \Delta_x u(x,r)+u_{rr}(x,r)+\frac{a}{r}u_r(x,r)=0,&(x,r)\in\mathbb{R}^n\times[0,+\infty),\\
    u(x,0)=f(x),&x\in\mathbb{R}^n.
\end{cases}
\end{aligned}  
\end{equation}
Recognizing that the dimensional parameter $a$ in \eqref{extensionfunction_a} can be analytically continued beyond integer values, they proved that
\begin{align*}
    C(-\Delta)^{\frac{1-a}{2}}f(x)=-\lim_{r\to0}r^au_r(x,r)=(a-1)\lim_{r\to0}\frac{u(x,r)-u(x,0)}{r^{1-a}}.
\end{align*}
Letting $a=1-s$, we get
\begin{align}
\label{extensionrelation_s}
    C(-\Delta)^{\frac{s}{2}}f(x)=-\lim_{r\to0}r^{1-s}u_r(x,r)=-s\lim_{r\to0}\frac{u(x,r)-u(x,0)}{r^{s}}.
\end{align}
Thus, this means that up to a constant, applying the fractional Laplacian to a function on the boundary is equivalent to taking the deformed derivative of its extension with respect to \eqref{extensionfunction_a} at the boundary.

There are two ways to prove \eqref{extensionrelation_s}. First, we observe that
\eqref{extensionfunction_a} is equivalent to
\begin{align*}
\begin{cases}
    \text{div}(r^{1-s}\nabla u(x,r))=0,&(x,r)\in\mathbb{R}^n\times[0,+\infty),\\
    u(x,0)=f(x),&x\in\mathbb{R}^n,
\end{cases}
\end{align*}
which is the minimizer for the functional
\begin{align}
\label{functional_Archimedean}
    J(u)=\int_{\mathbb{R}^n\times[0,+\infty)}|\nabla u|^2r^{1-s}dX.
\end{align}
Since fractional Laplacian is defined by the Fourier transform, a natural idea is to prove the corresponding energy functionals coincide by taking the Fourier transform of \eqref{extensionfunction_a}, then we can get
\begin{align*}
    \int_{\mathbb{R}^n\times[0,+\infty)}|\nabla u|^2r^{1-s}dX=\int_{\mathbb{R}^n}|\xi|^s|\hat{f}(\xi)|^2d\xi.
\end{align*}
Another way to prove \eqref{extensionrelation_s} is to find the Poisson kernel for \eqref{extensionfunction_a}, and relate $u$ with $f$ by the Poisson integral formula.

After their pioneering work, Stinga and Torrea generalized the extension method from the fractional Laplacian to a wide class of second-order differential operators and established a crucial Harnack inequality, greatly enhancing the universality of their method \cite{MR2754080}. Subsequently, Galé, Miana and Stinga generalized this framework to the generators of $C_0$-semigroups in Banach spaces. By synthesizing methodologies from semigroup theory and wave equations, they established a profound connection between the continuation problem and the core analytical structures of functional analysis \cite{MR3056307}. Subsequently, Chamorro and Jarrín promoted the results of Stinga and Torrea \cite{MR3348985}.

Frank-González-Monticelli-Tan introduced the extension framework into the geometric setting for the first time, and proved the continuation problem for the CR fractional Laplacian on the Heisenberg group and for sub-Laplacians on nilpotent Lie groups \cite{MR3286532}. Subsequently, Papageorgiou analyzed the large-time asymptotic behavior of the Caffarelli–Silvestre extension solution on Riemannian manifolds with nonnegative Ricci curvature \cite{MR4779594}.

Chen-Lei-Wei investigated the extension problems related to the higher order fractional Laplacian, and they gave a new proof to the dissipative a priori estimate of quasi-geostrophic equations by using this technique \cite{MR3773814}. Later, Cora and Musina revealed the equivalence between the higher-order fractional Laplacian and the s-polyharmonic extension operator, providing a novel perspective for the study of higher-order nonlocal problems \cite{MR4429579}. Besides, De Luca, Felli and Siclari rigorously proved for the first time the strong unique continuation property for the spectral fractional Laplacian, and classified the boundary singularities by using the extension method and the Almgren-type monotonicity formula \cite{MR4608212}. In addition, Biswas and Stinga promoted the extension method to arbitrary higher-order fractional powers of operators with $s>0$ in Banach spaces, resolved the issue of initial conditions required for establishing well-posedness, and provided a new class of subordination formulas \cite{MR4742773}.

However, there are no such results in the non-Archimedean case, so our goal is to find the $p$-adic analogue of this problem. Specifically, we focus on the extension problems on Bruhat-Tits tree
\begin{align*}
    T_p=PGL(2,\mathbb{Q}_p)/PGL(2,\mathbb{Z}_p),
\end{align*}
which is a $p+1$-regular tree and can be viewed as the analogue of a hyperbolic upper half-plane in Archimedean case, see in Figure \ref{fig:bruhat-tits_tree}.
\begin{figure}[H]
        \centering
        \includegraphics[width=0.6\textwidth]{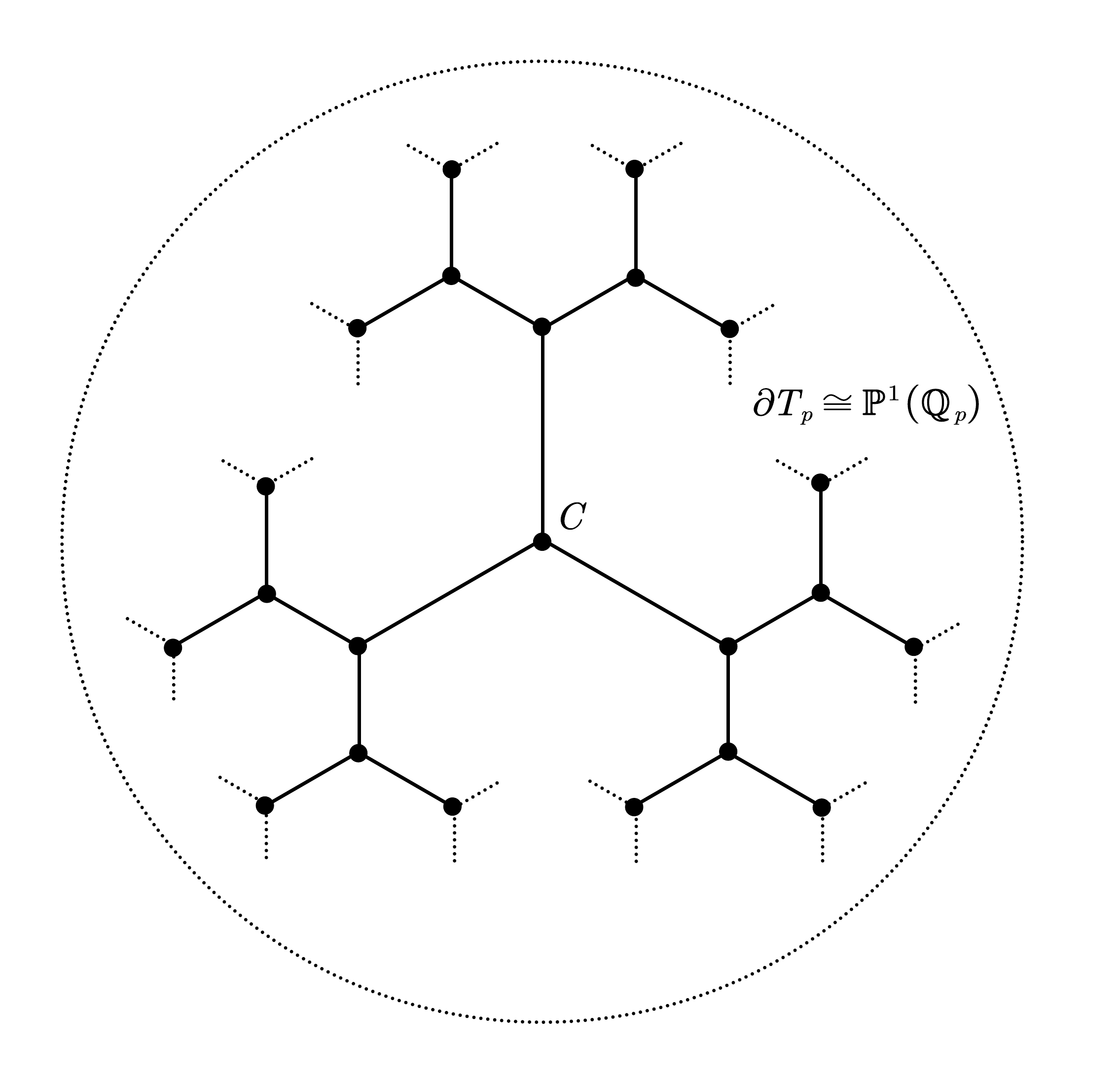}
        \caption{Bruhat-Tits tree $T_p$ for $p=2$}
        \label{fig:bruhat-tits_tree}
    \end{figure}

First, we fix some notations. Let $C$ denote a fixed reference vertex, serving as the base point within the tree $T_p$. Consider interior vertices $i,j\in T_p$ and boundary points $x,y\in\partial T_p$. The canonical graph distance $d(i,j)$ is defined as the number of edges comprising the unique simple path connecting $i$ and $j$. And let
\begin{align*}
    \delta(i_1\to i_2,j_1\to j_2)
\end{align*}
denote the length of the common part of the two paths for $i_1,i_2,j_1,j_2\in\overline{T}_p$, taken with a positive sign if they are oriented in the same direction and with a negative sign if they are oriented in opposite directions. 

After defining the distance between two vertices and two paths, we define
\begin{align*}
    \langle i,x\rangle=\delta(C\to x,C\to i)+\delta(i\to x,C\to i)
\end{align*}
to measure the distance between the vertex on the tree and the point on the boundary.

For the distance between two points on the boundary, we define
\begin{align*}
    |x,y|_p=p^{-\delta(C\to x,C\to y)},
\end{align*}
one can check that the distance is $GL(2,\mathbb{Z}_p)$-invariant, which can induce a $GL(2,\mathbb{Z}_p)$-invariant measure $\mu_0$ on the boundary defined by
\begin{align*}
    \mu_0(\partial B_i)=p^{-d(C,i)},
\end{align*}
where $B_i=\{j\in T_p|\text{we have the path }C\to i\to j\}$.

Recognizing the homeomorphism $\partial T_p \cong \mathbb{P}^1(\mathbb{Q}_p)$, the boundary naturally acts as the one-point compactification of the $p$-adic field. Consequently, the invariant measure $\mu_0$ on the tree boundary relates to the standard Haar measure $dx$ on $\mathbb{Q}_p$ via a stereographic transformation. Since each point $x\in\partial T_p$ is equivalent to a $p$-adic number, it has a unique $p$-adic expansion, and it also corresponds to the point at infinity of a ray emanating from a point $C_{-(l+1)}$ in $T_p$. Thus the ray $C_{-(l+1)}\to x$ uniquely determines its $p$-adic expansion. Specifically, if $x=\sum_{n=-l}^{\infty}a_np^n$ for $a_{-l}\neq 0$, then it can be viewed as a ray starting at $C_{-(l+1)}$, and at the first branch choosing the number $a_{-l}\in\{1,\dots,p-1\}$, after that at each branch choosing the number $a_n\in\{0,1,\dots,p-1\}$. It's easy to check that
\begin{equation}
\label{distance_relation}
\begin{aligned}
    |x,y|_p=\begin{cases}
        |x-y|_p,&|x|_p\leq1,\ |y|_p\leq1;\\
        |x^{-1}-y^{-1}|_p,&|x|_p>1,\ |y|_p>1;\\
        1,&\text{otherwise},
    \end{cases}
\end{aligned}    
\end{equation}

and
\begin{equation}
\label{measure_relation}
\begin{aligned}
    d\mu_0(x)=\begin{cases}
        dx,&|x|_p\leq1;\\
        \frac{dx}{|x|_p^2},&|x|_p>1.
    \end{cases}
\end{aligned}    
\end{equation}

Let $\hat{\Delta}_p$ be the standard graph Laplacian on $T_p$ defined by
\begin{align*}
    \hat{\Delta}_p\varphi(i)=\sum_{j\sim i}(\varphi(j)-\varphi(i)),
\end{align*}
and define the normal derivative at $\partial T_p$ as
\begin{align}
\label{normal_derivative_tree}
    \partial_n^{(p)}\varphi(x)=\lim_{i\to x}(\varphi(x)-\varphi(i))p^{d(C,i)}.
\end{align}

Given the totally disconnected topological nature of the $p$-adic field, which precludes the existence of canonical local differential operators. In 1990, Vladimirov defined the Vladimirov derivative via Fourier transform as
\begin{align*}
    D^su(x)=-\frac{1}{\Gamma_p(-s)}\int_{\mathbb{Q}_p}\frac{u(x)-u(z)}{|x-z|_p^{1+s}}dz,
\end{align*}
where $\Gamma_p(\alpha)=\frac{1-p^{\alpha-1}}{1-p^{-\alpha}}$ is the $p$-adic Gamma function \cite{MR1092528}.
\begin{remark}
    The generalised Vladimirov derivative $D^s$ can be viewed as the analogue of the fractional Laplacian $(-\Delta)^{\frac{s}{2}}$ as the symbol of the former is $|\xi|_p^s$ and that of the latter is $|\xi|^s$.
\end{remark}

Considering
\begin{align}
\label{eq:laplace_equation}
    \begin{cases}
        \hat{\Delta}_p\varphi(i)=0,&i\in T_p,\\
        \varphi|_{\partial T_p}=f(x).
    \end{cases}
\end{align}
Within the framework of non-Archimedean string theory, Zabrodin (1989) explicitly derived the Poisson kernel associated with the standard graph Laplacian $\hat{\Delta}_p$, demonstrating it to be $P(i,x) = \frac{p}{p+1}p^{\langle i,x\rangle}$, and he showed that
\begin{align}
\label{relation_s_1}
    \partial_n^{(p)}f(x)=\frac{p}{p+1}\int_{\partial T_p}\frac{f(x)-f(y)}{|x,y|_p^2}d\mu_0(y)
\end{align}
by the Poisson integral formula \cite{MR1003429}. Two years later, Gille made rigorous Zabrodin’s argument on the Poisson integral formula \cite{MR1715932}.

If we rewrite \eqref{relation_s_1} by \eqref{distance_relation} and \eqref{measure_relation}, we obtain
\begin{align}
\label{eq:extension_s=1}
    \partial_n^{(p)}f(x)=-\frac{p}{p+1}\Gamma_p(-1)\max\{1,|x|_p^2\}Df(x)=p^{-1}\max\{1,|x|_p^2\}Df(x),
\end{align}
and by the Green's formula on $T_p$ \cite{MR1003429}, we have
\begin{align*}
    \frac{1}{2}\sum_{\substack{i,j\\i\sim j}}(g(i)-g(j))(h(i)-h(j))=-\sum_{i}g(i)\hat{\Delta}_ph(i)+(p-1)\int_{\partial T_p}g(x)\partial_{n}^{(p)}h(x)d\mu_0(x),
\end{align*}
so that if we take $g=h=\varphi$ which satisfies \eqref{eq:laplace_equation}, and by \eqref{eq:extension_s=1} we have
\begin{align*}
     \frac{1}{2}\sum_{\substack{i,j\\i\sim j}}(\varphi(i)-\varphi(j))^2=&(p-1)\int_{\partial T_p}f(x)\partial_{n}^{(p)}\varphi(x)d\mu_0(x)\\
     =&(1-p^{-1})\int_{\partial T_p}f(x)\max\{1,|x|_p^2\}Df(x)d\mu_0(x)\\
     =&(1-p^{-1})\int_{\mathbb{Q}_p}f(x)Df(x)dx,
\end{align*}
which gives the relationship between the internal energy and the boundary energy after identifying $\partial T_p$ with $\mathbb{Q}_p$, and it recovers the Vladimirov derivative, which is the extension problem related to the Vladimirov derivative. Motivated by this elegant correspondence at $s=1$, this study embarks on generalizing the energy equivalence framework to encompass arbitrary fractional powers $s>0$.

Several years earlier, Heydeman, Marcolli, Saberi and Stoica first thought of the potential possibilities of this issue \cite{MR3858021}. Recently, the “scale corrected” harmonic extension on the Tate curve $T_{p^d}^{(w)}:=T_{p^d}/\sim_w$ has been used to study the finite-temperature extension of the $p$-adic AdS/CFT \cite{MR4882500}, which can be adapted to our setting.

Analogous to the Archimedean case, we have two possible ways to show the extension relationship. As there is no local derivative on $p$-adic field, we can't analogize \eqref{extensionfunction_a} to find our extension operator, and it is somewhat difficult to make the Fourier transform on the tree compatible with that on the boundary, so we will find the Poisson kernel and use the Poisson integral formula to solve this problem. Similarly, we define the deformed normal derivative on the boundary as
\begin{align}
\label{deformed_derivative}
    \partial_{n,s}^{(p)}\varphi(x)=\lim_{i\to x}(\varphi(x)-\varphi(i))p^{d(C,i)s},
\end{align}

Our main results are as follows:
\begin{theorem}
Considering the extension problem:
\begin{align*}
    \begin{cases}
        \tilde{D}^s\varphi(i)=0,&i\in T_p,\\
        \varphi|_{\partial T_p}=f(x),
    \end{cases}
\end{align*}
where $\tilde{D}^s$ is a weighted graph Laplacian on $T_p$ with the weight $w_{ij}=p^{d(C,i)(s-1)}+p^{d(C,j)(s-1)}$ when $i\sim j$. Then the relationship between the internal energy and the boundary energy is
    \begin{align*}
        &\sum_{i}\sum_{j\sim i}(\varphi(j)-\varphi(i))^2p^{d(C,i)(s-1)}\\
        =&\frac{(1-p^{-s})(p^{s-1}-p^{1-s})}{(1+p^{-1})(1-p^{-1-s})}
    \left(\int_{\partial T_p}f(x)d\mu_0(x)\right)^2+(1-p^{-s})(1+p^{1-s})\int_{\mathbb{Q}_p}a(x)f(x)D^s\left(a(x)f(x)\right)dx,
    \end{align*}
where
\begin{align*}
        a(x)=\begin{cases}
            1,&x\in\mathbb{Z}_p;\\
            |x|_p^{s-1},&x\in\mathbb{Q}_p\backslash\mathbb{Z}_p.
        \end{cases}
    \end{align*}
\end{theorem}
\begin{remark}
    If the mean value of $f(x)$ is zero, then we have
\begin{align*}
    \sum_{i}\sum_{j\sim i}(\varphi(j)-\varphi(i))^2p^{d(C,i)(s-1)}=(1-p^{-s})(1+p^{1-s})\int_{\mathbb{Q}_p}a(x)f(x)D^s\left(a(x)f(x)\right)dx,
\end{align*}
\end{remark}
\begin{remark}
    When $s=1$, the result coincides with Zabrodin's work \cite{MR1003429}. And we cannot connect generalised Vladimirov derivative with deformed normal derivative on the boundary for $s\neq1$, because the integral kernel has a constant shift and the power $1+s$ does not coincide with the power $2$ induced by the measure on tree. 
\end{remark}
\begin{remark}
    Our results can be generalised to finite extensions of $\mathbb{Q}_p$.
\end{remark}

\section{Extension problem for the generalised Vladimirov derivative}
\subsection{Define the energy functional}
We begin by constructing a discrete analogue of \eqref{functional_Archimedean} to derive the energy functional on $T_p$. 

First, the tree $T_p$ can be viewed as the discrete topological analogue of the upper half-space $\mathbb{R}^n\times[0,+\infty)$, and it is easy to analogize $|\nabla u|^2$ by $\displaystyle\sum_{j\sim i}(\varphi(j)-\varphi(i))^2$. Besides, we note that $r$ is the distance from $X=(x,r)$ to $\partial(\mathbb{R}^n\times[0,+\infty))=\mathbb{R}^n\times\{0\}$, so we need to find out the "distance" from the vertex on tree $T_p$ to its boundary. From \eqref{normal_derivative_tree}, we can see that when $i\to x$, $p^{-d(C,i)}=p^{-\langle i,x\rangle}$ can be viewed as the distance from $i$ to $\partial T_p$. So we treat $p^{-d(C,i)}$ as the analogue of $r$.

Therefore, we define the energy functional as
\begin{align}
\label{energy_functional_analogue}
    S[\varphi]=\frac{1}{2}\sum_{i}\sum_{j\sim i}(\varphi(j)-\varphi(i))^2p^{-d(C,i)(1-s)}.
\end{align}

\subsection{Find the Euler–Lagrange equation}
We rewrite the energy functional \eqref{energy_functional_analogue} as
\begin{align}
\label{energy_functional}
    S[\varphi]=\frac{1}{2}\sum_{i}\sum_{j\sim i}(\varphi(j)-\varphi(i))^2p^{d(C,i)(s-1)}.
\end{align}
Take any locally constant function $v$ such that $v|_{\partial T_p}=0$. Then the minimizer $\varphi$ of \eqref{energy_functional} satisfies
\begin{align*}
    \frac{dS[\varphi+\varepsilon v]}{d\varepsilon}\Big|_{\varepsilon=0}=&\sum_{i}\sum_{j\sim i}(\varphi(j)-\varphi(i))(v(j)-v(i))p^{d(C,i)(s-1)}\\
    =&\sum_{i}\sum_{j\sim i}(\varphi(j)v(j)-\varphi(i)v(i)-v(i)(\varphi(j)-\varphi(i))-\varphi(i)(v(j)-v(i)))p^{d(C,i)(s-1)}\\
    =&\sum_{i}(\hat{\Delta}_p(\varphi(i)v(i))-v(i)\hat{\Delta}_p\varphi(i)-\varphi(i)\hat{\Delta}_pv(i))p^{d(C,i)(s-1)}\\
    =&\sum_{i}(\varphi(i)v(i)\hat{\Delta}_pp^{d(C,i)(s-1)}-p^{d(C,i)(s-1)}v(i)\hat{\Delta}_p\varphi(i)-v(i)\hat{\Delta}_p(p^{d(C,i)(s-1)}\varphi(i)))\\
    =&\sum_{i}(\varphi(i)\hat{\Delta}_pp^{d(C,i)(s-1)}-p^{d(C,i)(s-1)}\hat{\Delta}_p\varphi(i)-\hat{\Delta}_p(p^{d(C,i)(s-1)}\varphi(i)))v(i)\\
    =&0,
\end{align*}
so we have the Euler–Lagrange equation
\begin{align*}
    &\tilde{D}^s\varphi(i)\\
    =&\varphi(i)\hat{\Delta}_pp^{d(C,i)(s-1)}-p^{d(C,i)(s-1)}\hat{\Delta}_p\varphi(i)-\hat{\Delta}_p(p^{d(C,i)(s-1)}\varphi(i))\\
    =&\sum_{j\sim i}\left(\varphi(i)(p^{d(C,j)(s-1)}-p^{d(C,i)(s-1)})-p^{d(C,i)(s-1)}(\varphi(j)-\varphi(i))-(p^{d(C,j)(s-1)}\varphi(j)-p^{d(C,i)(s-1)}\varphi(i))\right)\\
    =&\sum_{j\sim i}(p^{d(C,i)(s-1)}+p^{d(C,j)(s-1)})(\varphi(i)-\varphi(j))\\
    =&0,
\end{align*}
which is a weighted graph Laplacian on $T_p$, the weight $w_{ij}=p^{d(C,i)(s-1)}+p^{d(C,j)(s-1)}$ when $i\sim j$.

\subsection{Determine the Poisson kernel}
The Poisson kernel $P_s(i,x)$ with respect to $\tilde{D}^s$ satisfies
\begin{equation}
\label{Poisson_eq}
    \begin{aligned}
        \begin{cases}
            \tilde{D}^sP_s(i,x)=0,&i\in T_p,\\
            \displaystyle\lim_{i\to y}P_s(i,x)=\delta_x(y),&y\in\partial T_p.
        \end{cases}
    \end{aligned}
\end{equation}

Let $i_0$ be the branch point of the ray $C\to x$ and directed segment $C\to i$, and when $i\in C\to x$, $i_0=i$. So we have $\delta(C\to x,C\to i)=d(C,i_0)$ and $\delta(i\to x,C\to i)=-d(i_0,i)$. Furthermore, if we fix $x\in\partial T_p$, then $\forall i\in T_p$, $i$ can be uniquely determined by $d(C,i_0)$ and $d(i_0,i)$, so the solution depend only on
\begin{align*}
    \begin{cases}
        d(C,i_0):=m,\\
        d(i_0,i):=r.
    \end{cases}
\end{align*}
Let $P_s(i,x)=f(m,r)$. We now solve \eqref{Poisson_eq} by separation of variables. Assume $f(m,r)=A(m)B(r)$. Now we determine $A(m)$, $B(r)$ and the initial condition in the following three cases, respectively. For convenience, we denote $\alpha=s-1$.
\begin{enumerate}
    \item[]\textbf{Case 1.} $i\notin C\to x$ (Determine $B(r)$).
    Then from Figure \ref{case1}, the equation becomes
    \begin{align*}
    &(p^{\alpha(m+r)}+p^{\alpha(m+r-1)})(A(m)B(r)-A(m)B(r-1))\\
    &+p(p^{\alpha(m+r)}+p^{\alpha(m+r+1)})(A(m)B(r)-A(m)B(r+1))=0.
\end{align*}
Dividing both sides of the equation by $p^{\alpha(m+r)}A(m)$ and simplify, we get
\begin{align*}
    B(r+1)-(1+p^{-(\alpha+1)})B(r)+p^{-(\alpha+1)}B(r-1)=0,
\end{align*}
which is a second-order difference equation, whose characteristic equation is
\begin{align*}
    \lambda^2-(1+p^{-(\alpha+1)})\lambda+p^{-(\alpha+1)}=0,
\end{align*}
then
\begin{align*}
    (\lambda-1)(\lambda-p^{-(\alpha+1)})=0,
\end{align*}
so we have $\lambda=1$ or $\lambda=p^{-(\alpha+1)}$. Given the boundary condition $\displaystyle\lim_{i \to y} P_s(i,x) = 0$ for all $y \neq x$, the discrete maximum principle mandates that the solution must decay as the radial depth $r \to \infty$. Consequently, we rigorously discard the non-decaying root $\lambda = 1$, yielding the radial profile $B(r) = p^{-(\alpha+1)r}$.

    \begin{figure}[H]
        \centering
        \includegraphics[width=0.6\textwidth]{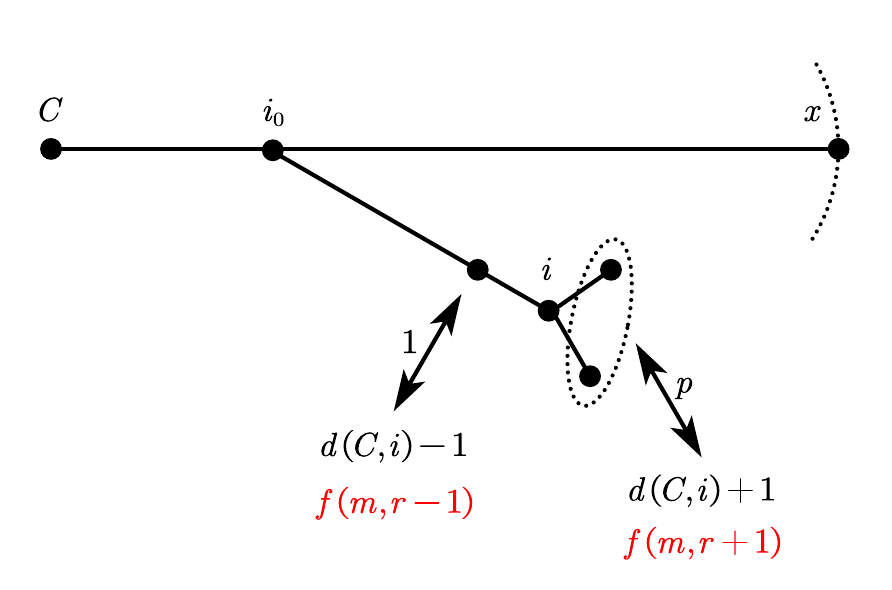}
        \caption{$i\notin C\to x$}
        \label{case1}
    \end{figure}
    \item[]\textbf{Case 2.} $i\in C\to x$ and $i\neq C$ (Determine $A(m)$). In this case, $r=0$. Then from Figure \ref{case2}, the equation becomes
    \begin{align*}
    &(p^{\alpha(m)}+p^{\alpha(m-1)})(A(m)-A(m-1))+(p^{\alpha(m)}+p^{\alpha(m+1)})(A(m)-A(m+1))\\
    &+(p-1)(p^{\alpha(m)}+p^{\alpha(m+1)})(A(m)-A(m)p^{-(\alpha+1)})=0,
\end{align*}
by simplifying, we obtain
\begin{align*}
    A(m+1)-(p^{-(\alpha+1)}+p)A(m)+p^{-\alpha}A(m-1)=0,
\end{align*}
which is also a second-order difference equation, whose characteristic equation is
\begin{align*}
    \mu^2-(p^{-(\alpha+1)}+p)\mu+p^{-\alpha}=0,
\end{align*}
then
\begin{align*}
    (\mu-p)(\mu-p^{-(\alpha+1)})=0,
\end{align*}
so we have $\mu=p$ or $\mu=p^{-(\alpha+1)}$. Then
\begin{align*}
    A(m)=C_1p^m+C_2p^{-(\alpha+1)m}.
\end{align*}
\begin{figure}[H]
        \centering
        \includegraphics[width=0.6\textwidth]{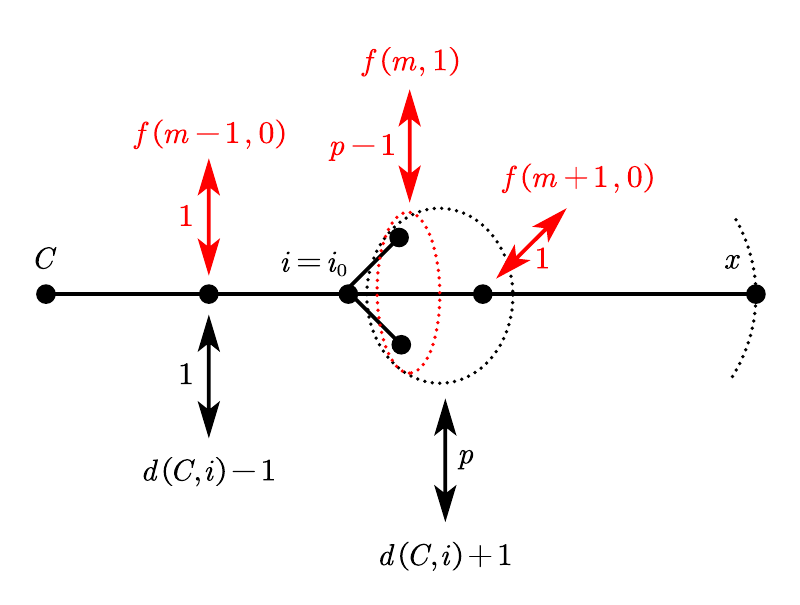}
        \caption{$i\in C\to x$ and $i\neq C$}
        \label{case2}
    \end{figure}
    \item[]\textbf{Case 3.} $i=C$ (Determine the initial value). In this case $m=r=0$. Then from Figure \ref{case3}, the equation becomes
    \begin{align*}
    (1+p^\alpha)(A(0)-A(1))+p(1+p^{\alpha})(A(0)-A(0)p^{-(\alpha+1)})=0,
\end{align*}
by simplification, we obtain
\begin{align*}
    A(1)=(1+p-p^{-\alpha})A(0),
\end{align*}
since $f(m.r)=A(m)B(r)=(C_1p^m+C_2p^{-(\alpha+1)m})p^{-(\alpha+1)r}$, substituting the initial condition, we obtain
\begin{align*}
    \begin{cases}
        C_1+C_2=A(0),\\
        C_1p+C_2p^{-(\alpha+1)}=(1+p-p^{-\alpha})A(0),
    \end{cases}
\end{align*}
solving this system yields
\begin{equation}
\label{C_1andC_2}
    \begin{aligned}
    \begin{cases}
        C_1=\frac{(p+1)(1-p^{-(\alpha+1)})}{p-p^{-(\alpha+1)}}A(0),\\
        C_2=\frac{p^{-\alpha}-1}{p-p^{-(\alpha+1)}}A(0).
    \end{cases}
\end{aligned}
\end{equation}

\begin{figure}[H]
        \centering
        \includegraphics[width=0.6\textwidth]{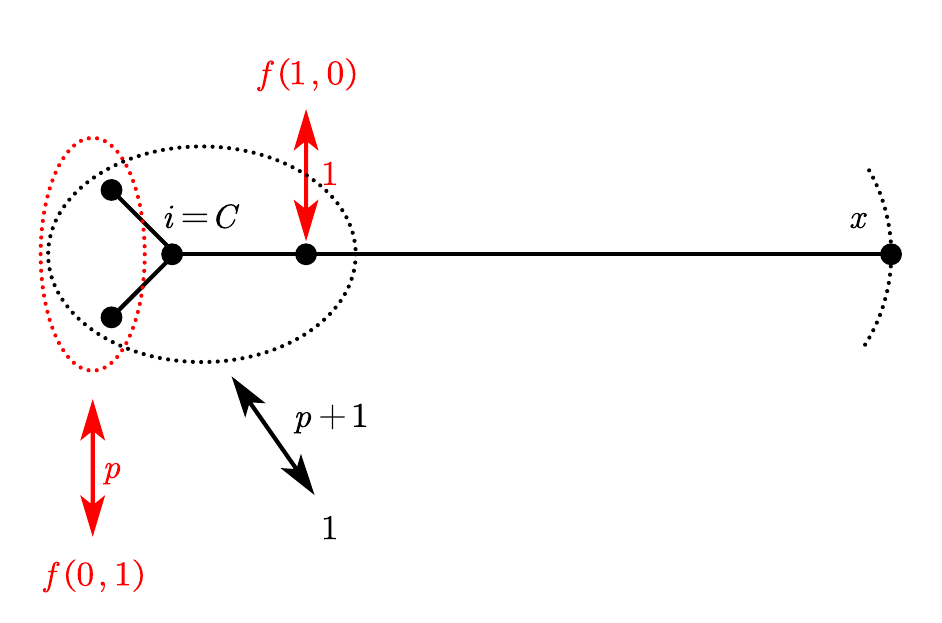}
        \caption{$i=C$}
        \label{case3}
    \end{figure}
\end{enumerate}

Substituting $\alpha=s-1$, we have
\begin{align*}
    P_s(i,x)=&A(0)\left(\frac{(p+1)(1-p^{-s})}{p-p^{-s}}p^{d(C,i_0)}+\frac{p^{1-s}-1}{p-p^{-s}}p^{-d(C,i_0)s}\right)p^{-d(i_0,i)s}\\
    =&A(0)\left(\frac{(p+1)(1-p^{-s})}{p-p^{-s}}p^{d(C,i_0)}+\frac{p^{1-s}-1}{p-p^{-s}}p^{-d(C,i_0)s}\right)p^{-(d(C,i)-d(C,i_0))s}\\
    =&A(0)\left(\frac{(p+1)(1-p^{-s})}{p-p^{-s}}p^{d(C,i_0)(1+s)}+\frac{p^{1-s}-1}{p-p^{-s}}\right)p^{-d(C,i)s}.
\end{align*}

Finally we determine $C_1$ and $C_2$. Since $P_s(i,x)$ is the Poisson kernel with respect to \eqref{Poisson_eq} and $\partial T_p$ is compact, for any locally constant function $f(x)$ on $\partial T_p$, we have
\begin{align*}
    \lim_{i\to x}\int_{\partial T_p}P_{s}(i,y)f(y)d\mu_0(y)=f(x).
\end{align*}
In particular, we take $f(x)=\chi_{\partial B_w}(x)$ for $w\in T_p$, it suffices to compute $F_w(i)=\int_{\partial B_w}P_{s}(i,y)d\mu_0(y)$. To evaluate the integral over the arbitrary boundary ball $\partial B_w$, it is imperative to stratify the computation based on the relative topological configuration of the evaluation vertex $i$ and the reference domain vertex $w$ with respect to the origin $C$. This geometric stratification guarantees that the piecewise constant nature of the Poisson kernel over distinct sub-branches is accurately integrated.
\begin{enumerate}
    \item[]\textbf{Case 1.} $i\in B_w$. In this case we have the path $C\to w\to i$. Let $i_k$ be the points on the path from $w$ to $i$, then $d(C,i_k)=k$, $k=d(C,w),d(C,w)+1,\dots,d(C,i)$. See Figure \ref{caseI}. Then $i_k$ has $p+1-2=p-1$ branch points not lying on path $C\to i$ for $k=d(C,w),d(C,w)+1,\dots,d(C,i)-1$, denoting each branch point by $i_{k,l}$, then $\mu_0(\partial B_{i_{k,l}})=p^{-{k+1}}$.

    When $x\in\partial B_{i_{k,l}}$, then $d(C,i_0)=d(C,i_k)=k$ and the Poisson kernel is constant on this branch. When $x\in\partial B_i$, the Poisson kernel is also constant on this branch. And we notice that $\mu_0(\partial B_{i})=p^{-d(C,i)}$, so we have
    
    \begin{align*}
    F_w(i)=&\sum_{k=d(C,w)}^{d(C,i)-1}(p-1)p^{-(k+1)}(C_1p^{k(1+s)}+C_2)p^{-d(C,i)s}+p^{-d(C,i)}(C_1p^{d(C,i)(1+s)}+C_2)p^{-d(C,i)s}\\
    =&(p-1)p^{-1-d(C,i)s}\sum_{k=d(C,w)}^{d(C,i)-1}(C_1p^{ks}+C_2p^{-k})+C_1+C_2p^{-d(C,i)(1+s)}\\
    =&(p-1)p^{-1-d(C,i)s}(C_1\frac{p^{d(C,i)s}-p^{d(C,w)s}}{p^s-1}+C_2\frac{p^{-d(C,w)}-p^{-d(C,i)}}{1-p^{-1}})+C_1+C_2p^{-d(C,i)(1+s)}\\
    =&(1-p^{-1})C_1\frac{1-p^{(d(C,w)-d(C,i))s}}{p^s-1}+C_2p^{-d(C,w)-d(C,i)s}+C_1.
\end{align*}
\begin{figure}[H]
        \centering
        \includegraphics[width=0.6\textwidth]{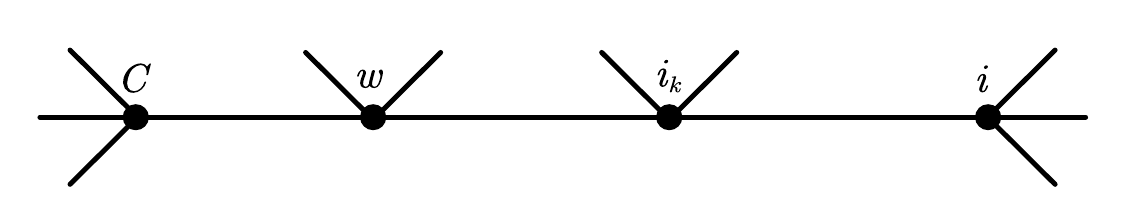}
        \caption{$i\in B_w$}
        \label{caseI}
    \end{figure}
    \item[]\textbf{Case 2.} $i\notin B_w$. We have three relative positions of $C,w,i$ in this case, see Figure \ref{caseII1}, Figure \ref{caseII2} and Figure \ref{caseII3}. But among all of these cases, the Poisson kernel is constant on $\partial B_w$ and $\mu_0(\partial B_w)=p^{-d(C,w)}$.
    
    So when $C,w,i$ are not on the same path,
    \begin{align*}
    F_w(i)=&p^{-d(C,w)}(C_1p^{d(C,i_0)(1+s)}+C_2)p^{-d(C,i)s}\\
    =&(C_1p^{d(C,i_0)(1+s)}+C_2)p^{-d(C,w)-d(C,i)s},
\end{align*}
    when we have the path $C\to i\to w$,
    \begin{align*}
    F_w(i)=&p^{-d(C,w)}(C_1p^{d(C,i_0)(1+s)}+C_2)p^{-d(C,i)s}\\
    =&(C_1p^{d(C,i)(1+s)}+C_2)p^{-d(C,w)-d(C,i)s},
\end{align*}
    when we have the path $i\to C\to w$,
    \begin{align*}
    F_w(i)=&p^{-d(C,w)}(C_1p^{d(C,i_0)(1+s)}+C_2)p^{-d(C,i)s}\\
    =&(C_1p^{-d(C,i)(1+s)}+C_2)p^{-d(C,w)-d(C,i)s},
\end{align*}
\begin{figure}[H]
        \centering
        \includegraphics[width=0.6\textwidth]{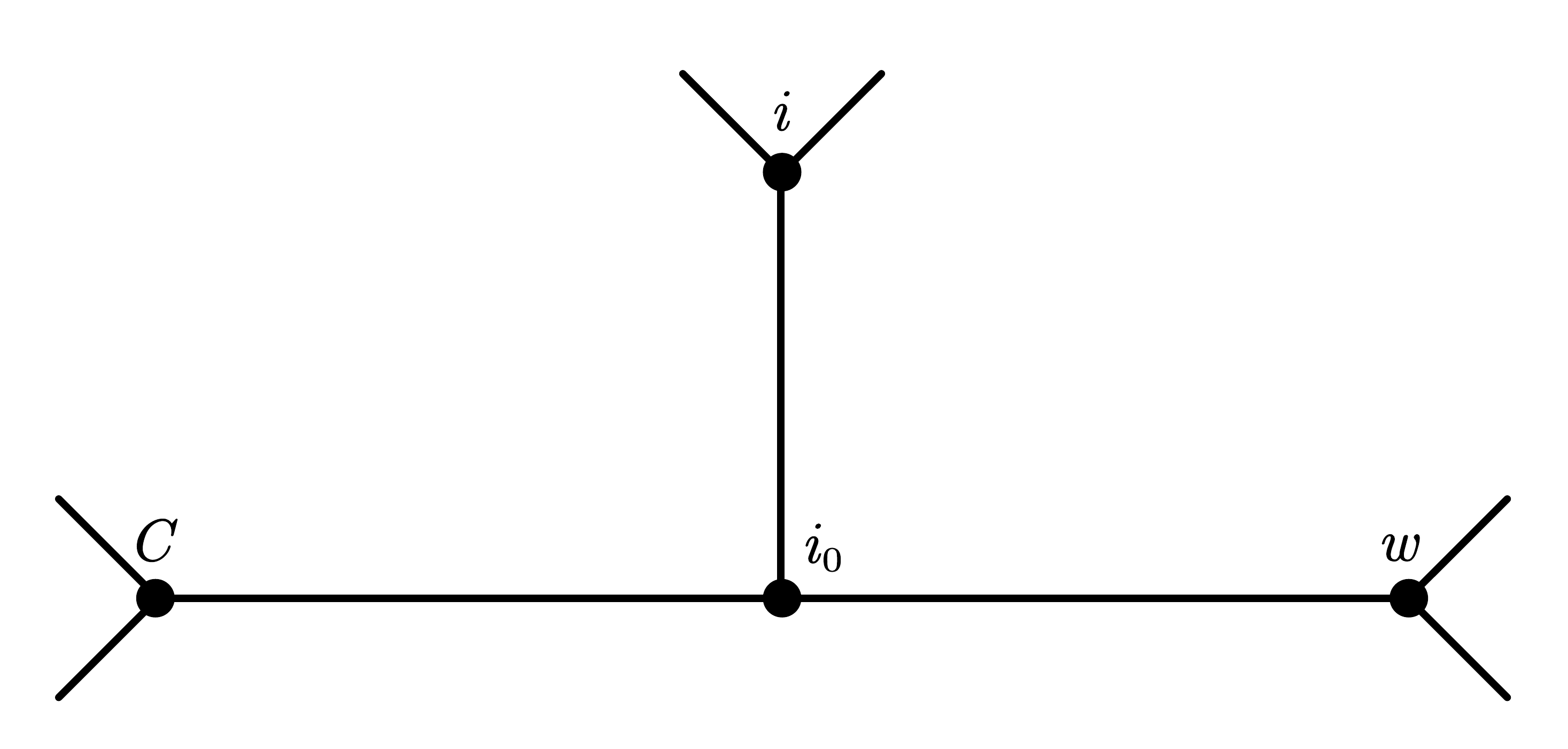}
        \caption{$i\notin B_w$ and $C,w,i$ are not on the same path}
        \label{caseII1}
    \end{figure}
\begin{figure}[H]
        \centering
        \includegraphics[width=0.6\textwidth]{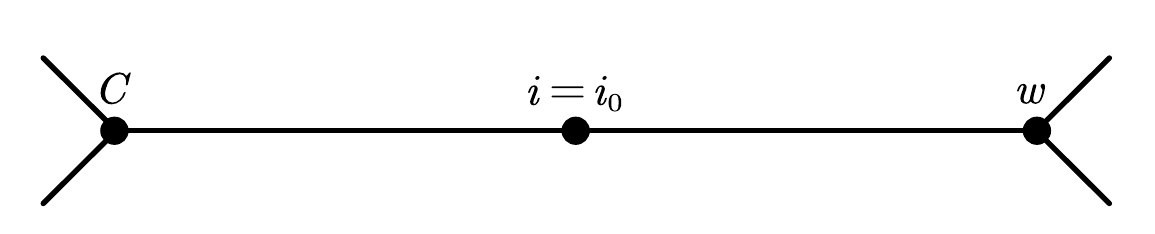}
        \caption{$i\notin B_w$ and we have the path $C\to i\to w$}
        \label{caseII2}
    \end{figure}
\begin{figure}[H]
        \centering
        \includegraphics[width=0.6\textwidth]{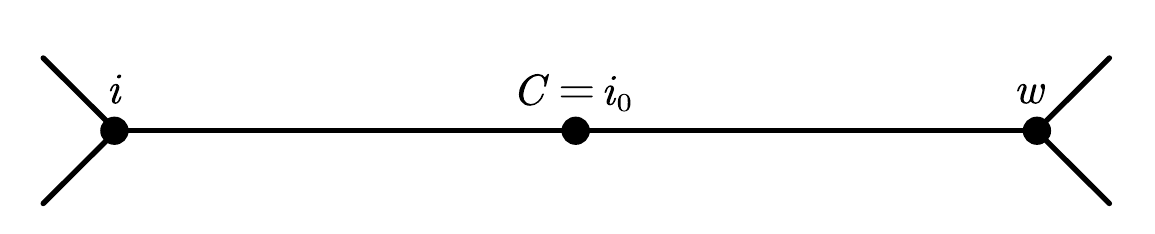}
        \caption{$i\notin B_w$ and we have the path $i\to C\to w$}
        \label{caseII3}
    \end{figure}
\end{enumerate}

In summary, we have    
\begin{align*}
    F_w(i)=\begin{cases}
        (1-p^{-1})C_1\frac{1-p^{(d(C,w)-d(C,i))s}}{p^s-1}+C_2p^{-d(C,w)-d(C,i)s}+C_1,&i\in B_w;\\
        (C_1p^{d(C,i_0)(1+s)}+C_2)p^{-d(C,w)-d(C,i)s},&i\notin B_w\ \text{and}\ i\notin C\to w;\\
        (C_1p^{d(C,i)(1+s)}+C_2)p^{-d(C,w)-d(C,i)s},&i\notin B_w\ \text{and}\ C\to i\to w;\\
        (C_1p^{-d(C,i)(1+s)}+C_2)p^{-d(C,w)-d(C,i)s},&i\notin B_w\ \text{and}\ i\to C\to w.
    \end{cases}
\end{align*}
Taking $i\to\partial T_p$, the case $i\notin B_w$ with the path $C\to i\to w$ will not appear and we have $d(C,i)\to+\infty$, so we have
\begin{align*}
    F_w(i)\to\begin{cases}
        \frac{(1-p^{-1})}{p^s-1}C_1+C_1=\frac{p^s-p^{-1}}{p^s-1}C_1,&i\to x;\\
        0,&i\to y\neq x,
    \end{cases}
\end{align*}
thus
\begin{align}
\label{C_1}
    \frac{p^s-p^{-1}}{p^s-1}C_1=1\Rightarrow C_1=\frac{p^s-1}{p^s-p^{-1}},
\end{align}
and from \eqref{C_1andC_2} we have
\begin{align}
\label{eq:C_2/C_1}
    \frac{C_2}{C_1}=\frac{p^{1-s}-1}{(p+1)(1-p^{-s})},
\end{align}
so
\begin{align}
\label{C_2}
    C_2=\frac{p^{1-s}-1}{(p+1)(1-p^{-s})}\frac{p^s-1}{p^s-p^{-1}}=\frac{p-p^s}{(p+1)(p^s-p^{-1})}.
\end{align}
Therefore, the Poisson kernel for the operator $\tilde{D}^s$ satisfying \eqref{Poisson_eq} is
\begin{align*}
    P_s(i,x)=\left(\frac{p^s-1}{p^s-p^{-1}}p^{d(C,i_0)(1+s)}+\frac{p-p^s}{(p+1)(p^s-p^{-1})}\right)p^{-d(C,i)s}.
\end{align*}

Lastly, since $\tilde{D}^s$ is a weighted graph Laplacian, the Poisson kernel we found is the unique solution of \eqref{Poisson_eq} by the maximum principle.
\begin{remark}
    When $s=1$, $\tilde{D}^1=-2\hat{\Delta}_p$, and $P_1(i,x)=\frac{p-1}{p-p^{-1}}p^{2d(C,i_0)-d(C,i)}=\frac{p}{p+1}p^{\langle i,x\rangle}$, which precisely coincides with the classical Poisson kernel for the unweighted graph Laplacian $\hat{\Delta}_p$ initially discovered by Zabrodin \cite{MR1003429}.
\end{remark}
\subsection{Relation with generalised Vladimirov derivative}
Considering the extension problem
\begin{equation}
\label{extension_problem}
    \begin{aligned}
        \begin{cases}
            \tilde{D}^s\varphi(i)=0,&i\in\ T_p,\\
            \varphi|_{\partial T_p}=f(x).
        \end{cases}
    \end{aligned}
\end{equation}
Then by the Poisson integral formula, we have
\begin{align*}
    \varphi(i)=\int_{\partial T_p}P_{s}(i,y)f(y)d\mu_0(y),
\end{align*}
and we notice that
\begin{align*}
    \varphi(x)=\int_{\partial T_p}P_{s}(i,y)f(x)d\mu_0(y)
\end{align*}
as $\int_{\partial T_p}P_{s}(i,y)d\mu_0(y)=1$, so
\begin{align*}
    \varphi(x)-\varphi(i)=\int_{\partial T_p}P_{s}(i,y)(f(x)-f(y))d\mu_0(y).
\end{align*}
Then the deformed normal derivative of $\varphi(i)$ at $\partial T_p$ is
\begin{equation}
\label{calculate_deformed_derivative}
\begin{aligned}
    \partial_{n,s}^{(p)}\varphi(x)=&\lim_{i\to x}(\varphi(x)-\varphi(i))p^{d(C,i)s}\\
    =&\lim_{i\to x}\int_{\partial T_p}\left(C_1p^{d(C,i_0)(1+s)}+C_2\right)p^{-d(C,i)s}(f(x)-f(y))d\mu_0(y)p^{d(C,i)s}\\
    =&\lim_{i\to x}\int_{\partial T_p}\left(C_1p^{d(C,i_0)(1+s)}+C_2\right)(f(x)-f(y))d\mu_0(y)\\
    =&\int_{\partial T_p}\left(\frac{C_1}{|x,y|_p^{1+s}}+C_2\right)(f(x)-f(y))d\mu_0(y),
\end{aligned}    
\end{equation}
where $C_1$ and $C_2$ can refer to \eqref{C_1} and \eqref{C_2}.

Since the power $1+s$ does not coincide with the power $2$ generated from the relationship between the measure on $\partial T_p$ and $\mathbb{Q}_p$, and the constant $C_2\neq0$ when $s\neq1$, so we cannot recover the generalised Vladimirov derivative $D^s$ through the deformed normal derivative on $\partial T_p$ unless $s=1$.

\section{Relationship between interior and boundary energies}
\begin{lemma}[Green’s first identity for weighted graphs]
\label{Green_identity_weighted_graph}
    Let $G=(V,E,w)$. $C\in V$ and denote $V_r=\{i\in V|d(C,i)\leq r\}$, $\partial V_r=\{i\in V|d(C,i)=r\}$, then for any $g,h:V\to\mathbb{R}$, we have
    \begin{align*}
        \frac{1}{2}\sum_{\substack{i,j\in V_r\\i\sim j}}(g(i)-g(j))(h(i)-h(j))w_{ij}=-\sum_{i\in V_r}g(i)\Delta_wh(i)-\sum_{i\in\partial V_r}g(i)\sum_{\substack{j\sim i\\j\notin V_r}}(h(i)-h(j))w_{ij}.
    \end{align*}
\end{lemma}
\begin{proof}
    Consider
    \begin{align}
    \label{Green_pf}
        \sum_{i\in V_r}g(i)\sum_{j\sim i}(h(i)-h(j))w_{ij}.
    \end{align}
    \begin{enumerate}
        \item[]\textbf{Case 1.} $j\in V_r$, then $\{i,j\}$ is an interior edge. We have the pair $g(i)(h(i)-h(j))w_{ij}$ and $g(j)(h(j)-h(i))w_{ij}$. Adding both yields $(g(i)-g(j))(h(i)-h(j))w_{ij}$, so this part contributes
        \begin{align*}
            \frac{1}{2}\sum_{\substack{i,j\in V_r\\i\sim j}}(g(i)-g(j))(h(i)-h(j))w_{ij}.
        \end{align*}
        \item[]\textbf{Case 2.} $j\notin V_r$, then $\{i,j\}$ is a boundary edge, in this case $i\in \partial V_r$ and $j\notin V_r$, so we only have $g(i)(h(i)-h(j))w_{ij}$. This part contributes
        \begin{align*}
            \sum_{i\in\partial V_r}g(i)\sum_{\substack{j\notin V_r\\j\sim i}}(h(i)-h(j))w_{ij}.
        \end{align*}
    \end{enumerate}
    Note that \eqref{Green_pf} is
        \begin{align*}
            -\sum_{i\in V_r}g(i)\Delta_wh(i),
        \end{align*}
    so we have
        \begin{align*}
            -\sum_{i\in V_r}g(i)\Delta_wh(i)=\frac{1}{2}\sum_{\substack{i,j\in V_r\\i\sim j}}(g(i)-g(j))(h(i)-h(j))w_{ij}+\sum_{i\in\partial V_r}g(i)\sum_{\substack{j\notin V_r\\j\sim i}}(h(i)-h(j))w_{ij}.
        \end{align*}
\end{proof}

\begin{theorem}[Green’s first identity for weighted $T_p$]
\label{Green_identity_weighted_T_p}
    Suppose $T_p$ has the weight $w_{ij}=p^{d(C,i)(s-1)}+p^{d(C,j)(s-1)}$ when $i\sim j$. Then we have
    \begin{align*}
        &\frac{1}{2}\sum_{\substack{i,j\\i\sim j}}(g(i)-g(j))(h(i)-h(j))(p^{d(C,i)(s-1)}+p^{d(C,j)(s-1)})\\
        =&-\sum_{i}g(i)\tilde{D}^sh(i)+(p-p^{1-s})(1+p^{s-1})\int_{\partial T_p}g(x)\partial_{n,s}^{(p)}h(x)d\mu_0(x).
    \end{align*}
\end{theorem}
\begin{proof}
    From Lemma \ref{Green_identity_weighted_graph}, we have
    \begin{align*}
        &\frac{1}{2}\sum_{\substack{i,j\in V_r\\i\sim j}}(g(i)-g(j))(h(i)-h(j))(p^{d(C,i)(s-1)}+p^{d(C,j)(s-1)})\\
        =&-\sum_{i\in V_r}g(i)\tilde{D}^sh(i)-\sum_{i\in\partial V_r}g(i)\sum_{\substack{j\sim i\\j\notin V_r}}(h(i)-h(j))(p^{d(C,i)(s-1)}+p^{d(C,j)(s-1)}).
    \end{align*}
    As $r\to+\infty$, the first term tends to
    \begin{align*}
        \frac{1}{2}\sum_{\substack{i,j\\i\sim j}}(g(i)-g(j))(h(i)-h(j))(p^{d(C,i)(s-1)}+p^{d(C,j)(s-1)}),
    \end{align*}
    and the second term tends to
    \begin{align*}
        -\sum_{i}g(i)\tilde{D}^sh(i).
    \end{align*}
    For the third term, as $i\in\partial V_r$, so $d(C,i)=r$, $d(C,j)=r+1$. Then the weight $w_{ij}=p^{r(s-1)}+p^{(r+1)(s-1)}=p^{-r}p^{rs}(1+p^{s-1})$. And from \eqref{deformed_derivative} we get
    \begin{align*}
        h(i)=h(x)-\partial_{n,s}^{(p)}h(x)p^{-rs}+o(p^{-rs}),
    \end{align*}
    and there are $p$ such vertices $j$ satisfying
    \begin{align*}
        h(j)=h(x)-\partial_{n,s}^{(p)}h(x)p^{-(r+1)s}+o(p^{-(r+1)s}),
    \end{align*}
    so
    \begin{align*}
        h(i)-h(j)=-\partial_{n,s}^{(p)}h(x)(p^{-rs}-p^{-(r+1)s})+o(p^{-rs}),
    \end{align*}
    thus
    \begin{align*}
        -(h(i)-h(j))w_{ij}=&\partial_{n,s}^{(p)}h(x)(p^{-rs}-p^{-(r+1)s})p^{-r}p^{rs}(1+p^{s-1})+o(p^{-r})\\
        =&\partial_{n,s}^{(p)}h(x)(1-p^{-s})(1+p^{s-1})p^{-r}+o(p^{-r}).
    \end{align*}
    Notice that $\mu_0(\partial B_i)=p^{-r}$, when $r\to+\infty$, $i\to x$, so the third term tends to
    \begin{align*}
        p(1-p^{-s})(1+p^{s-1})\int_{\partial T_p}g(x)\partial_{n,s}^{(p)}h(x)d\mu_0(x),
    \end{align*}
    so we have
    \begin{align*}
        &\frac{1}{2}\sum_{\substack{i,j\\i\sim j}}(g(i)-g(j))(h(i)-h(j))(p^{d(C,i)(s-1)}+p^{d(C,j)(s-1)})\\
        =&-\sum_{i}g(i)\tilde{D}^sh(i)+(p-p^{1-s})(1+p^{s-1})\int_{\partial T_p}g(x)\partial_{n,s}^{(p)}h(x)d\mu_0(x).
    \end{align*}
\end{proof}

\begin{lemma}
\label{lem:direct_calculation}
    \begin{align*}
        \int_{\mathbb{Q}_p\backslash\mathbb{Z}_p}\frac{|x|_p^{s-1}-|y|_p^{s-1}}{|x-y|_p^{1+s}}dy=\frac{1-p^{-1}}{1-p^{-s}}\frac{1}{|x|_p^{1+s}}-\frac{1}{|x|_p^{2}},
    \end{align*}
    where $x\in\mathbb{Q}_p\backslash\mathbb{Z}_p$. And
    \begin{align*}
        \int_{\mathbb{Q}_p\backslash\mathbb{Z}_p}\frac{1-|y|_p^{s-1}}{|y|_p^{1+s}}dy=\frac{1-p^{-1}}{1-p^{-s}}-1.
    \end{align*}
\end{lemma}
\begin{proof}
    By direct calculation, suppose $v_p(x)=k<0$, we have
    \begin{align*}
        &\int_{\mathbb{Q}_p\backslash\mathbb{Z}_p}\frac{|x|_p^{s-1}-|y|_p^{s-1}}{|x-y|_p^{1+s}}dy\\
        =&\sum_{l=-\infty}^{k-1}\frac{|x|_p^{s-1}-|y|_p^{s-1}}{|x-y|_p^{1+s}}dy\\
        =&\sum_{l=-\infty}^{k-1}\frac{p^{-k(s-1)}-p^{-l(s-1)}}{p^{-l(1+s)}}p^{-l}(1-p^{-1})+\sum_{l=k+1}^{-1}\frac{p^{-k(s-1)}-p^{-l(s-1)}}{p^{-k(1+s)}}p^{-l}(1-p^{-1})\\
        =&(1-p^{-1})\left(\sum_{l=1-k}^{\infty}\left(p^{k(1-s)}(p^{-s})^l-(p^{-1})^l\right)+\sum_{l=1}^{-k-1}\left(p^{2k}p^l-p^{k(s+1)}(p^s)^l\right)\right)\\
        =&(1-p^{-1})\left(\frac{p^{k-s}}{1-p^{-s}}-\frac{p^{k-1}}{1-p^{-1}}+\frac{p^{2k}-p^{k-1}}{p^{-1}-1}-\frac{p^{k(s+1)-p^{k-s}}}{p^{-s}-1}\right)\\
        =&(1-p^{-1})\left(\frac{p^{k(s+1)}}{1-p^{-s}}-\frac{p^{2k}}{1-p^{-1}}\right)\\
        =&\frac{1-p^{-1}}{1-p^{-s}}\frac{1}{|x|_p^{1+s}}-\frac{1}{|x|_p^{2}},
    \end{align*}
    and
    \begin{align*}
        \int_{\mathbb{Q}_p\backslash\mathbb{Z}_p}\frac{1-|y|_p^{s-1}}{|y|_p^{1+s}}dy=&\int_{\mathbb{Q}_p\backslash\mathbb{Z}_p}\frac{1}{|y|_p^{1+s}}-\frac{1}{|y|_p^{2}}dy\\
        =&\sum_{l=-\infty}^{-1}(p^{l(s+1)}-p^{2l})p^{-l}(1-p^{-1})\\
        =&(1-p^{-1})\sum_{l=1}^{\infty}\left((p^{-s})^l-(p^{-1})^l\right)\\
        =&(1-p^{-1})\left(\frac{p^{-s}}{1-p^{-s}}-\frac{p^{-1}}{1-p^{-1}}\right)\\
        =&\frac{1-p^{-1}}{1-p^{-s}}-1.
    \end{align*}
\end{proof}

\begin{lemma}
\label{lem:energy_relation_Qp_tree_boundary}
    Let
    \begin{align*}
        a(x)=\begin{cases}
            1,&x\in\mathbb{Z}_p;\\
            |x|_p^{s-1},&x\in\mathbb{Q}_p\backslash\mathbb{Z}_p.
        \end{cases}
    \end{align*}
    Then we have
    \begin{align*}
        &\int_{\mathbb{Q}_p\times\mathbb{Q}_p}\frac{(a(x)f(x)-a(y)f(y))^2}{|x-y|_p^{1+s}}dxdy\\
        =&\int_{\partial T_p\times\partial T_p}\frac{(f(x)-f(y))^2}{|x,y|_p^{1+s}}d\mu_0(x)d\mu_0(y)+2\left(\frac{1-p^{-1}}{1-p^{-s}}-1\right)\int_{\partial T_p}f^2(x)d\mu_0(x).
    \end{align*}
\end{lemma}
\begin{proof}
Decomposing $\mathbb{Q}_p$ into $\mathbb{Z}_p$ and $\mathbb{Q}_p\backslash\mathbb{Z}_p$ and substituting \eqref{distance_relation} and \eqref{measure_relation}, we have
\begin{equation}
\label{eq:diffference_between_energies}
\begin{aligned}
    &\int_{\mathbb{Q}_p\times\mathbb{Q}_p}\frac{(a(x)f(x)-a(y)f(y))^2}{|x-y|_p^{1+s}}dxdy-\int_{\partial T_p\times\partial T_p}\frac{(f(x)-f(y))^2}{|x,y|_p^{1+s}}d\mu_0(x)d\mu_0(y)\\
    =&\int_{\mathbb{Q}_p\backslash\mathbb{Z}_p\times\mathbb{Q}_p\backslash\mathbb{Z}_p}\frac{(|x|_p^{s-1}f(x)-|y|_p^{s-1}f(y))^2-|x|_p^{s-1}|y|_p^{s-1}(f(x)-f(y))^2}{|x-y|_p^{1+s}}dxdy\\
    &+\int_{\mathbb{Q}_p\backslash\mathbb{Z}_p\times\mathbb{Z}_p}\frac{(|x|_p^{s-1}f(x)-f(y))^2-|x|_p^{s-1}(f(x)-f(y))^2}{|x-y|_p^{1+s}}dxdy\\
    &+\int_{\mathbb{Z}_p\times\mathbb{Q}_p\backslash\mathbb{Z}_p}\frac{(f(x)-|y|_p^{s-1}f(y))^2-|y|_p^{s-1}(f(x)-f(y))^2}{|x-y|_p^{1+s}}dxdy\\
    =&\int_{\mathbb{Q}_p\backslash\mathbb{Z}_p}|x|_p^{s-1}f^2(x)dx\int_{\mathbb{Q}_p\backslash\mathbb{Z}_p}\frac{|x|_p^{s-1}-|y|_p^{s-1}}{|x-y|_p^{1+s}}dy+\int_{\mathbb{Q}_p\backslash\mathbb{Z}_p}|y|_p^{s-1}f^2(y)dy\int_{\mathbb{Q}_p\backslash\mathbb{Z}_p}\frac{|y|_p^{s-1}-|x|_p^{s-1}}{|x-y|_p^{1+s}}dx\\
    &+\int_{\mathbb{Q}_p\backslash\mathbb{Z}_p}(|x|_p^{2s-2}-|x|_p^{s-1})f^2(x)dx\int_{\mathbb{Z}_p}\frac{1}{|x-y|_p^{1+s}}dy+\int_{\mathbb{Z}_p}f^2(y)dy\int_{\mathbb{Q}_p\backslash\mathbb{Z}_p}\frac{1-|x|_p^{s-1}}{|x-y|_p^{1+s}}dx\\
    &+\int_{\mathbb{Z}_p}f^2(x)dx\int_{\mathbb{Q}_p\backslash\mathbb{Z}_p}\frac{1-|y|_p^{s-1}}{|x-y|_p^{1+s}}dy+\int_{\mathbb{Q}_p\backslash\mathbb{Z}_p}(|y|_p^{2s-2}-|y|_p^{s-1})f^2(y)dy\int_{\mathbb{Z}_p}\frac{1}{|x-y|_p^{1+s}}dx\\
    =&2\int_{\mathbb{Q}_p\backslash\mathbb{Z}_p}f^2(x)|x|_p^{s-1}dx\left(\int_{\mathbb{Z}_p}\frac{|x|_p^{s-1}-1}{|x|_p^{1+s}}dy+\int_{\mathbb{Q}_p\backslash\mathbb{Z}_p}\frac{|x|_p^{s-1}-|y|_p^{s-1}}{|x-y|_p^{1+s}}dy\right)\\
    &+2\int_{\mathbb{Z}_p}f^2(x)dx\int_{\mathbb{Q}_p\backslash\mathbb{Z}_p}\frac{1-|y|_p^{s-1}}{|y|_p^{1+s}}dy.
\end{aligned}    
\end{equation}
Noticing that
\begin{align*}
    \int_{\mathbb{Z}_p}\frac{|x|_p^{s-1}-1}{|x|_p^{1+s}}dy=\frac{1}{|x|_p^{2}}-\frac{1}{|x|_p^{1+s}}
\end{align*}
and by Lemma \ref{lem:direct_calculation}, \eqref{eq:diffference_between_energies} becomes
    \begin{align*}
        &2\int_{\mathbb{Q}_p\backslash\mathbb{Z}_p}f^2(x)|x|_p^{s-1}\left(\frac{1-p^{-1}}{1-p^{-s}}\frac{1}{|x|_p^{1+s}}-\frac{1}{|x|_p^{2}}+\frac{1}{|x|_p^{2}}-\frac{1}{|x|_p^{1+s}}\right)dx+2\int_{\mathbb{Z}_p}f^2(x)\left(\frac{1-p^{-1}}{1-p^{-s}}-1\right)dx\\
        =&2\left(\frac{1-p^{-1}}{1-p^{-s}}-1\right)\left(\int_{\mathbb{Z}_p}f^2(x)dx+\int_{\mathbb{Q}_p\backslash\mathbb{Z}_p}f^2(x)\frac{dx}{|x|_p^{2}}\right)\\
        =&2\left(\frac{1-p^{-1}}{1-p^{-s}}-1\right)\int_{\partial T_p}f^2(x)d\mu_0(x).
    \end{align*}
\end{proof}

\begin{remark}
    The piecewise function $a(x)$ fundamentally acts as a conformal weighting factor. Its operational necessity arises from the geometric imperative to dynamically compensate for the metric distortion introduced when mapping the spherically symmetric, compact boundary $\partial T_p$ onto the translationally invariant, non-compact field $\mathbb{Q}_p$ via stereographic projection. This conformal adjustment is essential for aligning the boundary measure $\mu_0$ with the standard Haar measure $dx$ while preserving the fractional scaling dimensions.
\end{remark}

\begin{theorem}
The relationship between the internal energy and the boundary energy is as follows:
    \begin{align*}
        &\sum_{i}\sum_{j\sim i}(\varphi(j)-\varphi(i))^2p^{d(C,i)(s-1)}\\
        =&\frac{(1-p^{-s})(p^{s-1}-p^{1-s})}{(1+p^{-1})(1-p^{-1-s})}
    \left(\int_{\partial T_p}f(x)d\mu_0(x)\right)^2+(1-p^{-s})(1+p^{1-s})\int_{\mathbb{Q}_p}a(x)f(x)D^s\left(a(x)f(x)\right)dx.
    \end{align*}
\end{theorem}

\begin{proof}
Applying Theorem \ref{Green_identity_weighted_T_p} to a function $\varphi$ that strictly satisfies the harmonic extension condition outlined in \eqref{extension_problem}, we deduce
\begin{align}
\label{eq:energy_original}
    \frac{1}{2}\sum_{\substack{i,j\\i\sim j}}(\varphi(i)-\varphi(j))^2(p^{d(C,i)(s-1)}+p^{d(C,j)(s-1)})=(p-p^{1-s})(1+p^{s-1})\int_{\partial T_p}\varphi(x)\partial_{n,s}^{(p)}\varphi(x)d\mu_0(x),
\end{align}
by \eqref{energy_functional} and \eqref{calculate_deformed_derivative}, \eqref{eq:energy_original} becomes
\begin{equation}
\label{eq:relation_interior_tree_boundary}
\begin{aligned}
    2S[\varphi]=&(p-p^{1-s})(1+p^{s-1})\int_{\partial T_p}\varphi(x)\partial_{n,s}^{(p)}\varphi(x)d\mu_0(x)\\
    =&(p-p^{1-s})(1+p^{s-1})\int_{\partial T_p}f(x)\int_{\partial T_p}\left(\frac{C_1}{|x,y|_p^{1+s}}+C_2\right)(f(x)-f(y))d\mu_0(y)d\mu_0(x)\\
    =&\frac{1}{2}(p-p^{1-s})(1+p^{s-1})\int_{\partial T_p\times\partial T_p}\left(\frac{C_1}{|x,y|_p^{1+s}}+C_2\right)(f(x)-f(y))^2d\mu_0(x)d\mu_0(y).
\end{aligned}
\end{equation}
Since
\begin{equation}
\label{eq:int_(fx-fy)^2}
\begin{aligned}
    \int_{\partial T_p\times\partial T_p}(f(x)-f(y))^2d\mu_0(x)d\mu_0(y)=&\int_{\partial T_p\times\partial T_p}\left(f^2(x)+f^2(y)-2f(x)f(y)\right)d\mu_0(x)d\mu_0(y)\\
    =&2(1+p^{-1})\int_{\partial T_p}f^2(x)d\mu_0(x)-2\left(\int_{\partial T_p}f(x)d\mu_0(x)\right)^2.
\end{aligned}    
\end{equation}
And by \eqref{eq:C_2/C_1}, we have
\begin{equation}
\label{eq:coeffiecent_int_f^2}
\begin{aligned}
    \left(\frac{1-p^{-1}}{1-p^{-s}}-1\right)-(1+p^{-1})\frac{C_2}{C_1}=&\frac{1-p^{-1}}{1-p^{-s}}-1-(1+p^{-1})\frac{p^{1-s}-1}{(p+1)(1-p^{-s})}\\
    =&\frac{1-p^{-1}}{1-p^{-s}}-1-\frac{p^{-s}-p^{-1}}{1-p^{-s}}\\
    =&0.
\end{aligned}    
\end{equation}
Then by \eqref{eq:int_(fx-fy)^2} and \eqref{eq:coeffiecent_int_f^2}, we obtain
\begin{align*}
    &\frac{C_2}{C_1}\int_{\partial T_p\times\partial T_p}(f(x)-f(y))^2d\mu_0(x)d\mu_0(y)\\
    =&2\left(\frac{1-p^{-1}}{1-p^{-s}}-1\right)\int_{\partial T_p}f^2(x)d\mu_0(x)-2\frac{C_2}{C_1}\left(\int_{\partial T_p}f(x)d\mu_0(x)\right)^2.
\end{align*}
Then by Lemma \ref{lem:energy_relation_Qp_tree_boundary}, we get
\begin{align*}
    &C_1\int_{\mathbb{Q}_p\times\mathbb{Q}_p}\frac{(a(x)f(x)-a(y)f(y))^2}{|x-y|_p^{1+s}}dxdy\\
    =&\int_{\partial T_p\times\partial T_p}\left(\frac{C_1}{|x,y|_p^{1+s}}+C_2\right)(f(x)-f(y))^2d\mu_0(x)d\mu_0(y)+2C_2\left(\int_{\partial T_p}f(x)d\mu_0(x)\right)^2.
\end{align*}
Substituting to \eqref{eq:relation_interior_tree_boundary} yields
\begin{align*}
    &2S[\varphi]\\
    =&\frac{1}{2}(p-p^{1-s})(1+p^{s-1})\left(C_1\int_{\mathbb{Q}_p\times\mathbb{Q}_p}\frac{(a(x)f(x)-a(y)f(y))^2}{|x-y|_p^{1+s}}dxdy-2C_2\left(\int_{\partial T_p}f(x)d\mu_0(x)\right)^2\right)\\
    =&(p-p^{1-s})(1+p^{s-1})\left(C_1\int_{\mathbb{Q}_p}a(x)f(x)dx\int_{\mathbb{Q}_p}\frac{a(x)f(x)-a(y)f(y)}{|x-y|_p^{1+s}}dy-C_2\left(\int_{\partial T_p}f(x)d\mu_0(x)\right)^2\right)\\
    =&(p-p^{1-s})(1+p^{s-1})\left(-C_1\Gamma_p(-s)\int_{\mathbb{Q}_p}a(x)f(x)D^s\left(a(x)f(x)\right)dx-C_2\left(\int_{\partial T_p}f(x)d\mu_0(x)\right)^2\right)\\
    =&\frac{(1-p^{-s})(p^{s-1}-p^{1-s})}{(1+p^{-1})(1-p^{-1-s})}
    \left(\int_{\partial T_p}f(x)d\mu_0(x)\right)^2+(1-p^{-s})(1+p^{1-s})\int_{\mathbb{Q}_p}a(x)f(x)D^s\left(a(x)f(x)\right)dx.
\end{align*}
\end{proof}

\begin{corollary}
If the mean value of $f(x)$ is zero, i.e. $\int_{\partial T_p}f(x)d\mu_0(x)=0$, then we have
\begin{align*}
    \sum_{i}\sum_{j\sim i}(\varphi(j)-\varphi(i))^2p^{d(C,i)(s-1)}=(1-p^{-s})(1+p^{1-s})\int_{\mathbb{Q}_p}a(x)f(x)D^s\left(a(x)f(x)\right)dx.
\end{align*}
\end{corollary}

\section*{Acknowledgements}
YJS thanks professor Jürgen Jost for his insightful idea and invaluable guidance during his visit to Shanghai Jiao Tong University. YJS also extends his sincere gratitude to Professor Bobo Hua for his constant support throughout this research.

\bibliographystyle{plain}
\bibliography{references}

\noindent An Huang, anhuang@brandeis.edu\\
\emph{Department of Mathematics, Brandeis University, Waltham, MA 02453, USA}\\[-8pt]

\noindent Yaojia Sun, 26110180043@m.fudan.edu.cn\\
\emph{School of Mathematical Sciences, Fudan University, Shanghai, 200433, P.R. China}\\[-8pt]
\end{document}